\documentclass[a4paper,12pt]{article}

\usepackage{amssymb,amsmath,amsthm}
\usepackage[all]{xy}

\usepackage[T1]{fontenc}
\usepackage[utf8]{inputenc}
\usepackage[backend=biber]{biblatex}
\usepackage[top=2cm,bottom=2cm,left=2cm,right=4cm,marginparsep=0.3cm,marginparwidth=3cm,includefoot]{geometry}
\usepackage{mathtools}
\usepackage{radon}
\usepackage[colorlinks=true, pdfstartview=FitV, linkcolor=blue, citecolor=blue, urlcolor=blue]{hyperref}
\usepackage[normalem]{ulem}  
\usepackage{mathrsfs}
\usepackage{accents}
\usepackage{mathbbol}
\usepackage{xspace}
\usepackage{marginnote}
\usepackage{tikz-cd}
\usepackage{tocvsec2}
\usepackage[framemethod=default]{mdframed}
\usepackage{tikz}
\usetikzlibrary{patterns}
\usepackage{graphicx}
\usepackage{stmaryrd}
\usepackage{geometry}
\usepackage{hyperref} 
\usepackage{tikz-cd}
\usepackage{stackengine}
\usepackage{scalerel}
\usepackage{stackengine, graphicx}
\usepackage{adjustbox}
\usetikzlibrary{fit}
\usepackage{yhmath}
\usepackage{enumitem}

\numberwithin{equation}{section}
\makeatletter
\newcommand{\widecheck}[1]{\mathpalette\widecheck@{#1}}

\newcommand{\widecheck@}[2]{%
	\setbox0=\hbox{$#1#2$}%
	\vbox{%
		\offinterlineskip
		\halign{%
			##\cr
			\hidewidth
			\raisebox{0.7ex}{%
				\rotatebox[origin=c]{180}{$#1\widehat{\phantom{#2}}$}%
			}%
			\hidewidth\cr
			\noalign{\kern-0.4ex}
			\box0\cr
		}%
	}%
}
\makeatother

\numberwithin{equation}{section}
\usepackage[colorlinks=true, pdfstartview=FitV, linkcolor=blue, citecolor=blue, urlcolor=blue]{hyperref}

\hypersetup{
	hypertexnames=false
}

\newtheorem{proposition-definition}{Proposition-definition}
\newtheorem*{theorem*}{Theorem}
\newtheorem*{proposition*}{Proposition}

\makeatletter
\renewcommand{\l@section}{\@dottedtocline{1}{1.5em}{2.6em}}
\renewcommand{\l@subsection}{\@dottedtocline{2}{4.0em}{3.6em}}
\renewcommand{\l@subsubsection}{\@dottedtocline{3}{7.4em}{4.5em}}
\makeatother

\UseRawInputEncoding

\DeclareMathAlphabet{\mathpzc}{OT1}{pzc}{m}{it}
\DeclareSymbolFontAlphabet{\amsmathbb}{AMSb}%
\DeclareMathOperator{\Rad}{Rad}

\newmdenv[
frametitle=Reminder,
skipabove=\topsep,
skipbelow=\topsep,
]{reminder}

\newmdenv[
frametitle=Problem,
skipabove=\topsep,
skipbelow=\topsep,
]{Problem}

\newcommand\restr[2]{{
		\left.\kern-\nulldelimiterspace
		#1
		\right|_{#2} 
}}

\title{Convolution invertibility for constructible sheaves and functions}
\author{Mehdi Benchoufi}
\date{\today}

\begin{document}
	
	\maketitle
	
	\begin{center}
		\emph{Dedicated to J.L.}
	\end{center}
	
	\begin{abstract}
		We characterize convolution-invertible, compactly supported piecewise-linear sheaves on a finite-dimensional real vector space. We show that they factor, up to a shift, as convolutions of constant sheaves on two bounded convex sets, one closed and the other relatively open.
		To this end, we first establish a more general characterization of convolution invertibility for compactly supported constructible functions. In the piecewise-linear case, this yields a convex factorization which lifts to sheaves. The key step is a geometric characterization of invertibility in the dual projective space via the Radon transform.
	\end{abstract}
	
	\setcounter{tocdepth}{2}
	\tableofcontents
	\bigskip
	
	\section{Introduction and main results}
	
	Let \(E\) be a real vector space of dimension \(n\geq1\), and let \(k=\mathbb R\) or \(\mathbb C\). We study convolution invertibility for compactly supported constructible sheaves and integer-valued constructible functions on \(E\). For \(F,G\in\Derb(k_E)\), convolution is defined by
	\eqn
	F\star G:=Rs_!(F\etens G),
	\eneqn
	where \(s:E\times E\to E\) is the addition map. An object in any of the convolution categories considered below will be called \(\star\)-invertible if it is invertible with respect to \(\star\). We denote convolution inverses by $F^{\star-1}$ for sheaves and by $\phi^{\star-1}$ for constructible functions.
	
	For sheaves, we prove two structural results. First, if \(S\subset E\) is a non-empty compact subanalytic subset, then \(k_S\) is invertible if and only if \(S\) is convex, by Theorem~\ref{th:main_theorem}. Second, by Theorem~\ref{th:description_inverse_star}, the inverse of any compactly supported invertible constructible sheaf \(F\) is \(D_E(F^a)\), where $D_E$ denotes the duality and $F^a$ the antipodal transform of $F$.
	
	We also consider the convolution ring \(\mathrm{CF}_c(E)\) of compactly supported constructible functions, with product \(\phi\star\psi:=\int_s(\phi\boxtimes\psi)\). The Euler--Poincar\'e morphism is compatible with convolution and duality~\cite[Th.~9.7.11]{KS90}, so invertible sheaves give invertible constructible functions. Our subsequent results give a geometrical characterization of all units of \(\mathrm{CF}_c(E)\).
	
	On a real line \(L\), these units are exactly the functions \(\pm\mathbf 1_I\), where \(I\) is a non-empty bounded open or closed interval. In arbitrary dimension, Theorem~\ref{th:invertible_constructible_functions} characterizes invertibility by the requirement that every one-dimensional projection satisfy \(\eim{u}\phi=\varepsilon_u\mathbf 1_{I_u}\), with \(\varepsilon_u\in\{-1,1\}\) and \(I_u\) such an interval. The inverse is again \(D_E(\phi^a)\).
	
	Theorem~\ref{th:main_theorem_cf} gives a convex factorization of every unit: a piecewise-linear function \(\phi\in\mathrm{CF}_c(E)\) is invertible if and only if there exist
	non-empty bounded convex sets \(C,D\subset E\),
	with \(C\) closed and \(D\) relatively open (i.e., open in its affine hull), such that
	\eqn
	\phi=\pm\mathbf 1_C\star\mathbf 1_{D}
	\eneqn
	The factors can be chosen polyhedral.
	
	The proof uses Schapira's inversion formula~\cite{Sch95} to translate the projection criterion into a geometric description in the dual projective space. After convolving $\phi$ with the indicator function of a sufficiently large closed ball, its Radon transform is, up to sign, the indicator function of a set which, in suitable affine charts, is the closed region between two continuous graphs over a hyperplane. 

	The factorization for constructible functions lifts to the derived category, which leads to the main result of this article. By Corollary~\ref{th:factorization_invertible_pl_sheaves}, a compactly supported PL constructible sheaf \(F\) is invertible if and only if there exist non-empty bounded convex sets \(C,D\subset E\),
	with \(C\) closed and \(D\) relatively open, and \(d\in\mathbb Z\), such that
	\eqn
	F\simeq k_C\star k_D[d]
	\eneqn
	
	\subsection*{Acknowledgments}
	The author would like to thank Pierre Schapira for many fruitful conversations and valuable remarks. No generative AI tools were used prior to the finalization of the proofs. ChatGPT was used during the validation of this manuscript for: (i) improving extensively the phrasing and readability (especially the introductory part of sections); (ii) verifying the validity of the proofs, during which it identified an error in Theorem~\ref{th:main_theorem} (a standalone special case of the main theorem that does not rely on the tools developed in this paper) - while it suggested an elementary proof, we chose to provide an alternative argument to highlight the role of the Radon transform in our approach; and (iii) identifying two relevant references, one concerning Theorem ~\ref{th:main_theorem} and another supporting an elementary geometric fact used in Lemma~\ref{lem:convex_cone_decomposition}, where it conveyed the idea of using the classical simplicial topology notion of vertex star and visibility kernel to improve the readibility of the proof. Also, it suggested in the end of a proof of Lemma~\ref{lem:geometric_series_non_compact} to refactor a computation using semi-continuity.
	\section{Preliminaries}
	
	We follow the notations of~\cite{KS90}. Let $X$ be a real analytic manifold. We denote by $\pi_X : T^*X \to X$ the cotangent bundle to \(X\), by $\dot{\pi}_X : \dot{T}^*X \to X$ the cotangent bundle with the zero-section removed, and by \(T^*_M X\) the conormal bundle to a submanifold \(M \subset X\).
	
	Let $k$ denote a field of caracteristic $0$. For a locally closed subset $Z$ of $X$, we denote by $k_Z$, the constant sheaf, constant on $Z$ with stalk $k$, and $0$ elsewhere. We denote by $\Derb(k_X)$ the bounded category of sheaves of $k$-modules on $X$, $\Derb_{\mathbb{R}-c}(k_X)$ its
	full triangulated subcategory of objects with $\mathbb{R}$-constructible cohomology sheaves. We write $\Derb_c(k_X)$ and $\Derb_{\mathbb{R}-c,c}(k_X)$ for their respective full triangulated subcategories consisting of objects with compact support.
	
	Let $Y$ be a real analytic manifold and $f$ a morphism $Y\to X$. We denote by $\oim{f},\opb{f},\eim{f},\epb{f},\hom,\tens$ the six Grothendieck operations.
	
	For $F\in\Derb(k_X)$, we denote by $D_X(F)$ the dual sheaf of $F$ and by $SS(F)$ its singular support. Given a point $x\in X$, we will call the constant sheaf on $\{x\}$ a Dirac sheaf, and denote it by $\delta_x$.
	
	Throughout, $E$ denotes a real vector space of dimension $n$. For \(F,G\in D^b(k_E)\), convolution is defined by
	\[
	F \star G := \reim{s} (F \boxtimes G)
	\]
	where \( s: E \times E \to E \) is the addition map.

	Let $a:E\to E$, $x\mapsto -x$, be the antipodal map. We write $F^a:=\oim{a}F$.
	
	We denote by \(\mathrm{CF}(X)\) the group of integer-valued constructible functions on \(X\), and by  \(\mathrm{CF}_c(X)\) the subgroup of compactly supported ones.

	A convex polyhedron is an intersection of finitely many open or closed affine half-spaces. Following \cite{KS18}, an object \(F\in\Derb(k_E)\) is called piecewise linear (PL) if there exists a locally finite covering of \(E\) by convex polyhedra on each of which all its cohomology sheaves are constant of finite rank. Similarly, a function \(\phi:E\to\mathbb Z\) is called PL if it is constant on each member of such a covering.
	
	For a non-empty subset $A$ of a finite-dimensional real affine space, we denote by $\mathring A$ its relative interior, that is, its interior in its affine hull $\operatorname{Aff}(A)$. We say that $A$ is relatively open if $A=\mathring A$.
	
	Let $L$ be a one-dimensional real vector space. We denote by $\mathcal L(E,L)$ the vector space of linear maps from $E$ to $L$, and set
	\[
	\mathcal L^\times(E,L):=\mathcal L(E,L)\setminus\{0\}.
	\]
	We refer to these non-zero maps as one-dimensional linear projections, and write $[u]$ for the class of $u$ modulo
	multiplication by a non-zero scalar.

	\subsection{Constructible functions}
	\label{sec:constructible_functions}
	
	We refer to~\cite{Sch95} for the theory of constructible functions and their integral transforms. Let \(M\) be a real analytic manifold. A function \(\phi:M\to\mathbb Z\) is constructible if each level set \(\phi^{-1}(m)\) is subanalytic and the family \(\{\phi^{-1}(m)\}_{m\in\mathbb Z}\) is locally finite. If \(A\subset M\) is subanalytic, we denote by \(\mathbf 1_A\) its indicator function.
	
	For a compactly supported constructible function $\varphi$ on $M$, choose a decomposition
	$\varphi=\sum_{i=1}^N c_i\mathbf 1_{K_i}$, with $c_i\in\mathbb Z$ and $K_i\subset M$ compact, subanalytic and contractible.
	Its Euler integral is defined by
	\[
	\int_M\varphi:=\sum_{i=1}^N c_i
	\]
	This integer is independent of the chosen decomposition (see \cite[\S~2]{Sch95}).
	
	For $\phi\in\mathrm{CF}(M)$, its dual $D_M\phi$ is defined by
	\[
	(D_M\phi)(x):=\int_M\phi\,\mathbf 1_{B_\varepsilon(x)},
	\qquad 0<\varepsilon\ll1,
	\]
	where $B_\varepsilon(x)$ is the open ball centred at $x$ in a local analytic coordinate chart.
	For sufficiently small $\varepsilon$, this value is independent of $\varepsilon$ and of the chosen chart.
	
	If \(f:N\to M\) is subanalytic and proper on the support of \(\psi\in\mathrm{CF}(N)\), its direct image is defined by
	\[
	\left(\int_f\psi\right)(x)
	:=
	\int_{f^{-1}(x)}\psi
	\]
	where the integrand is the restriction of $\psi$ to the fibre $f^{-1}(x)$.

	For \(\phi,\psi\in\mathrm{CF}_c(E)\), we define their convolution by
	\[
	\phi\star\psi
	:=
	\int_s(\phi\boxtimes\psi)
	\]
	where \(s:E\times E\to E\) is the addition map and \((\phi\boxtimes\psi)(x,y)=\phi(x)\psi(y)\).
	
	By~\cite[Th.~9.7.11]{KS90}, the local Euler--Poincar\'e characteristic
	\[
	F\longmapsto \bigl(x\mapsto\chi(F_x)\bigr)
	\]
	induces an isomorphism
	\[
	\chi:
	K_0\bigl(\Derb_{\mathbb R\text{-}c}(k_M)\bigr)
	\xrightarrow{\sim}
	\mathrm{CF}(M)
	\]
	where \(K_0\bigl(\Derb_{\mathbb R\text{-}c}(k_M)\bigr)\) denotes the Grothendieck ring of $\Derb_{\mathbb R\text{-}c}(k_M)$. This isomorphism is compatible with proper direct image,  inverse image, exterior
	product and duality.
	
	In particular, for compactly supported objects on a vector space \(E\), it induces an isomorphism of rings
	\[
	\chi:
	K_0\bigl(\Derb_{\mathbb R\text{-}c,c}(k_E),\star\bigr)
	\xrightarrow{\sim}
	\bigl(\mathrm{CF}_c(E),\star\bigr)
	\]
	One has
	\[
	\begin{aligned}
		\chi(F\star G) &= \chi(F)\star\chi(G)\\
		\chi(D_EF) &= D_E\chi(F)\\
	\end{aligned}
	\]
	
	\subsection{Reminder on the Radon transform}
	\label{reminder:radon_transform}
	
	We refer to~\cite{DS96a} for the sheaf-theoretic Radon transform and
	to~\cite{Sch95} for its counterpart for constructible functions.
	Let \(E\) be a real vector space of dimension \(n\), and let \(P\) be
	its projective compactification, with coordinates
	\((x_0:\ldots:x_n)\). We identify \(E\) with the standard affine chart
	\(\{x_0\neq0\}\subset P\). For \((x,\xi)\in E\times E^*\), we denote by
	\(\langle\xi,x\rangle\) the standard duality pairing.
	
	Let \(P^*\) denote the dual projective space. For \(\xi\in P^*\), we
	write \(\check{\xi}\subset P\) for the corresponding projective
	hyperplane. Whenever \(\check{\xi}\cap E\neq\varnothing\), we set
	\[
	\mathcal V_\xi:=\check{\xi}\cap E
	\]
	so that \(\mathcal V_\xi\) is an affine hyperplane in \(E\); we denote
	by \(V_\xi\subset E\) its associated vector space.
	
	Let
	\[
	v:P\times P^*\longrightarrow P^*\times P,
	\qquad
	v(x,\xi)=(\xi,x)
	\]
	be the canonical transposition map. We consider the incidence
	hypersurfaces
	\[
	S:=\{(x,\xi)\in P\times P^*;\ x\in\check{\xi}\},
	\qquad
	v(S):=\{(\xi,x)\in P^*\times P;\ x\in\check{\xi}\}.
	\]
	We use \(q_1\) and \(q_2\) for the projections onto the first and
	second factors of the corresponding product.
	
	For subsets \(A\subset P\) and \(B\subset P^*\), we define their
	set-Radon transforms by
	\[
	\widehat A
	:=
	q_2\bigl(q_1^{-1}(A)\cap S\bigr),
	\qquad
	\widecheck B
	:=
	q_2\bigl(q_1^{-1}(B)\cap v(S)\bigr)
	\]
	
	\subsubsection{Radon transform of sheaves}
	
	For \(F\in D^b(k_P)\) and \(G\in D^b(k_{P^*})\), the Radon transforms
	are defined by
	\[
	\Rad(F)
	:=
	Rq_{2!}\bigl(k_S[n-1]\otimes q_1^{-1}F\bigr),
	\qquad
	\Rad^*(G)
	:=
	Rq_{2!}\bigl(k_{v(S)}[n-1]\otimes q_1^{-1}G\bigr)
	\]
	
	The stalks of the Radon transform are given by
	\[
	\Rad(F)_\xi
	\simeq
	R\Gamma(\check{\xi};F)[n-1],
	\qquad
	\xi\in P^*
	\]
	
	We shall also use the notation
	\[
	\dot T^*P:=T^*P\setminus T^*_P P,
	\qquad
	\dot T^*P^*:=T^*P^*\setminus T^*_{P^*}P^*
	\]
	There is a canonical contact transformation
	\eq\label{morph:projective_contact_transform}
	\chi:\dot T^*P\xrightarrow{\sim}\dot T^*P^*
	\eneq
	and the Radon transform exchanges microsupports in the sense that
	\[
	SS(\Rad(F))\cap\dot T^*P^*
	=
	\chi\bigl(SS(F)\cap\dot T^*P\bigr)
	\]
	
	We shall use the following local description. Let
	\[
	I:=\{(x,\xi)\in E\times E^*;\ \langle\xi,x\rangle=0\}
	\]
	Notice that $I=S\cap (E\times E^*)$. 
	
	After identifying $T^*E\simeq E\times E^*$, $T^*E^*\simeq E^*\times E$, the map \(\chi\) is given on \((E\times E^*)\setminus I\) by
	\[
	\chi(x,\xi)
	=
	\left(
	\frac{1}{\langle\xi,x\rangle}\xi,
	\langle\xi,x\rangle x
	\right)
	\]
	
	\subsubsection{Radon transform of constructible functions}
	
	For \(\phi\in\mathrm{CF}(P)\) and \(\psi\in\mathrm{CF}(P^*)\), we define
	\[
	\Rad(\phi)
	:=
	\int_{q_2}
	\bigl(q_1^*\phi\cdot\mathbf 1_S\bigr),
	\qquad
	\Rad^*(\psi)
	:=
	\int_{q_2}
	\bigl(q_1^*\psi\cdot\mathbf 1_{v(S)}\bigr)
	\]
	
	Equivalently, for \(\xi\in P^*\),
	\[
	\Rad(\phi)(\xi)
	=
	\int_{\check{\xi}}\phi
	\]
	
	We have the compatibility formula
	\[
	\chi(\Rad(F))
	=
	(-1)^{n-1}\Rad(\chi(F))
	\]
	
	Assume that \(\dim P=n>0\). Schapira's inversion formula~\cite{Sch95}
	gives
	\eq\label{eq:schapira_radon_inversion}
	\Rad^*\Rad(\phi)
	=
	\begin{cases}
		\phi,
		& n \text{ odd},\\[0.3em]
		-\phi+
		\left(\displaystyle\int_P\phi\right)\mathbf 1_P,
		& n \text{ even}.
	\end{cases}
	\eneq

	\paragraph*{Examples of invertible sheaves}
	Constant sheaves on compact convex sets provide basic examples of $\star$-invertible sheaves.
	Theorem~\ref{th:main_theorem} will show that convexity is also necessary among non-empty compact subanalytic sets.
	
	\section{$\star$-Invertibility on the line}\label{sec:one_dimension_projections}
	
	We classify compactly supported convolution-invertible sheaves and constructible functions on the line. These classifications will serve as the starting point for the projection criteria in the next section.
	
	\subsection{Constructible sheaves}
	
	We recall the following structure theorem for constructible sheaves:
	\begin{theorem}[Guillermou~{\cite[Cor.~7.3]{Guill16}}; see also Crawley--Boevey~{\cite[Th.~1.1]{CB15}}]\label{th:structure_sheaf_dim1}
		Let $F$ be an object in $\Derb_{\mathbb{R}-c,c}(k_\mathbb{R})$.
		There exist a finite family of intervals \(\{I_\alpha\}_{\alpha\in A}\) and integers \(d_\alpha\) such that
		\[
		F \simeq \bigoplus_{\alpha\in A} k_{I_\alpha}[d_\alpha]
		\]
		Moreover, such a decomposition is unique (up to reindexing).
	\end{theorem}
	
	We can state the following:
	\begin{theorem}\label{th:invertible_sheaf_dim_1}
		Let $F$ be an object in $\Derb_{\mathbb{R}-c,c}(k_\mathbb{R})$. Then $F$ is \(\star\)-invertible if and only if there exist a bounded closed or open interval $I$, such that, up to a shift:
		$$
		F \simeq k_ I
		$$ 
	\end{theorem}
	\begin{proof}
		\emph{($\impliedby$)}: For $a<b$, the inverse of $k_{[a,b]}$ is $k_{]-b,-a[}[1]$. For $a=b$, the inverse of $k_{\{a\}}$ is $k_{\{-a\}}$.
		
		\emph{($\implies$)}: let $G\in\Derb_{\mathbb{R}-c}(k_\mathbb{R})$ be the inverse of $F$. We can decompose $G$ into the direct sum of its cohomologies. According to \ref{th:structure_sheaf_dim1}, there exist finite families of intervals \(\{I_\alpha\}_{\alpha\in A}\), \(\{I_\beta\}_{\beta\in B}\) and integers $d_\alpha,d_\beta$, such that $F \simeq \bigoplus_{\alpha\in A} k_{I_\alpha}[d_\alpha]$, $G \simeq \bigoplus_{\beta\in B} k_{I_\beta}[d_\beta]$ respectively. 
		
		Let us observe that constant sheaves on semi-open intervals are not invertible ($k_{[a,b[}\star k_{]c,d]}\simeq 0$). Now, $k_{\{0\}}\simeq F\star G \simeq \bigoplus_{\alpha\in A,\beta\in B}( k_{I_\alpha}\star k_{I_\beta})[d_\alpha+d_\beta]$ and the decomposition is unique (up to reindexing). For any $\alpha\in A$ such that $I_\alpha$ is semi-open, $I_\alpha\star I_\beta$ is still semi-open or possibly empty. In the latter case, if the convolution by some semi-open factor $I_\alpha$ vanishes, then all the factors of $G$ would be also semi-open, contradicting its non-invertibility. Hence, all the $I_\alpha,I_\beta$ are either open or closed. Moreover, by uniqueness, there are exactly two indices $\alpha\in A,\beta\in B$, such that $I_\alpha$ and $I_\beta$ are non-empty. Hence, $F \simeq k_{I_\alpha}[d_\alpha]$. Obviously, if $I_\alpha$ is closed, then $I_\beta$ is open, and conversely. Also those intervals are bounded since constant sheaves on such intervals are not invertible.
	\end{proof}
	
	\subsection{Constructible functions}
	
	Let \(L\) be a one-dimensional real vector space. We shall determine the group of
	units of \(\bigl(\mathrm{CF}_c(L),\star\bigr)\) and show that its elements are
	precisely, up to sign, the indicator functions of bounded open or closed intervals. 
	
	Fix an identification \(L\simeq\mathbb R\). We denote by
	\(\mathrm{CF}_{\mathrm{lb}}(L)\) the ring of constructible functions
	whose support is bounded from below. Since addition is proper on the product of two such supports, Euler convolution makes \(\mathrm{CF}_{\mathrm{lb}}(L)\) a commutative ring containing \(\mathrm{CF}_c(L)\).
	
	\begin{lemma}\label{lem:separated_interval_tail}
		Let \(I\subset L\) be a non-empty bounded interval, either open
		or closed, let \(\varepsilon\in\{-1,1\}\), and let
		\(\phi\in\mathrm{CF}_c(L)\) be non-zero. Assume that the supports of
		\(\mathbf 1_I\) and \(\phi\) are disjoint and that \(\supp(\phi)\) lies entirely
		on one side of \(I\). Then
		\[
		\varepsilon\mathbf 1_I+\phi
		\]
		is not invertible in \(\mathrm{CF}_c(L)\).
	\end{lemma}
	
	\begin{proof}
		Up to applying the antipodal map and up to a convolution by $\pm$ the indicator function of a bounded open or closed
		interval, which is invertible, we reduce to a function of the form $\delta_{0_L}-\phi$, where $\supp(\phi)\subset ]0,+\infty[$.
		
		In \(\mathrm{CF}_{\mathrm{lb}}(L)\), the series
		\[
		\psi:=\sum_{n\geq0}\phi^{\star n}
		\]
		is well defined. Indeed, the family of supports $\supp(\phi^{\star n})\bigr)_{n\geq0}$ is locally finite. Hence, $\psi$ is the inverse $\delta_{0_L}-\phi$ in \(\mathrm{CF}_{\mathrm{lb}}(L)\). We claim that \(\psi\) is not compactly supported and we prove it in the Lemma below.
		
		By uniqueness of the inverse in \(\mathrm{CF}_{\mathrm{lb}}(L)\), we conclude that $\delta_{0_L}-\phi$ is not  \(\star\)-invertible in $\mathrm{CF}_c(L)$.
	\end{proof}

	\begin{lemma}\label{lem:geometric_series_non_compact}
	Let \(\eta\in\mathrm{CF}_c(L)\) be non-zero and assume that $\supp(\eta)\subset]0,+\infty[$
		
	Then the locally finite geometric series $\psi:=\sum_{n\geq0}\eta^{\star n}$ does not have compact support.
	\end{lemma}
	
	\begin{proof}
		Assume that \(\psi\) has compact support, and set \(B:=w\). Then \(\supp(\psi)\subset[0,B]\). For \(m\geq1\), set $b_m:=\sup\supp(\eta^{\star m})=m\,\sup\supp(\eta)$. We may choose \(m\) such that $b_m>B$ and we set $C:=b_m$ and $C_0:=\max(B,b_1,\ldots,b_{m-1})$
		
		The finite geometric-series identity gives
		\eqn
		\psi = \sum_{0\leq n\leq m-1}\eta^{\star(n)} + \eta^{\star m}\star\psi
		\eneqn
		Hence,
		\eq\label{eq:vanishing_tail}
		\eta^{\star m}\star\psi=0
		\qquad\text{on } ]C_0,+\infty[
		\eneq 
		We will get a contradiction. In sufficiently small neighborhoods of the upper endpoints \(C\) and \(B\), we can write
		\eqn
		\eta^{\star m} = \lambda_1\delta_C + \mu_1\mathbf 1_{(c,C]},
		\qquad
		\psi = \lambda_2\delta_B + \mu_2\mathbf 1_{(b,B]}.
		\eneqn
		At \(C+B\), we have $\eta^{\star m}\star\psi = (\lambda_1+\mu_1)(\lambda_2+\mu_2)$. If this number is non-zero, this contradicts~(\ref{eq:vanishing_tail}). If both factors vanish, then, for \(\varepsilon>0\) sufficiently small, the
		extremal intersection at \(C+B\) is an open interval. Thus $\eta^{\star m}\star\psi$ evaluated at \(C+B-\varepsilon\) is $-\mu_1\mu_2\neq0$, contradicting again~(\ref{eq:vanishing_tail}).
		
		Assume finally that exactly one of the factors vanishes. We treat the case $\lambda_2+\mu_2=0$, the other one being symmetric.
		
		Starting from \(B\) and moving to the left, let \(b\) be the first point such that \(=\psi(b)-\psi(b+\varepsilon)\neq0\), for a sufficiently small $\varepsilon$. Such a point exists since otherwise \(\psi\) would be a sum of half-open pieces \([a,b[\), contradicting $\chi(\psi)\neq0$. Write
		\[
		\varphi:=
		\psi\mathbf1_{(b,+\infty)}+\psi(b+)\delta_{b}
		=\sum_i c_i\mathbf1_{[a_i,b_i[}.
		\]
		Set \(r_i:=c_i(h\star\mathbf1_{[a_i,b_i[})(C+b)\). Convolution with the indicator of a half-open interval \([a,b[\) is right-continuous. Hence
		\[
		0=(h\star\psi)(C+b)=(\psi(b)-\psi(b+)) h(C)+\sum_i r_i,
		\qquad
		0=(h\star\psi)(C+b+\varepsilon)=\sum_i r_i,
		\]
		indeed, since \(C+b\geq C>C_0\), \eqref{eq:vanishing_tail} garanties that both terms vanishes. This contradicts \(\alpha h(C)\neq0\). Therefore $\psi$ cannot have compact support.
	\end{proof}
	
	We can now state the following description of the units of \(\bigl(\mathrm{CF}_c(L),\star\bigr)\).
	
	\begin{proposition}\label{prop:invertible_constructible_functions_line}
		Let \(\phi\in\mathrm{CF}_c(L)\). Then \(\phi\) is \(\star\)-invertible if and only if there exist
		\(\varepsilon\in\{-1,1\}\) and a non-empty bounded interval \(I\subset L\), either open or closed, such that
		\[
		\phi=\varepsilon\mathbf 1_I.
		\]
		In this case,
		\[
		\phi^{\star-1}=D_L(\phi^a).
		\]
	\end{proposition}
	
	\begin{proof}
		\emph{($\impliedby$)}: For $a<b\in L$, the statement follows from the formulas
		\eqn
		\mathbf 1_{[a,b]}^{\star-1}
		=
		-\mathbf 1_{(-b,-a)},
		\qquad
		\mathbf 1_{(a,b)}^{\star-1}
		=
		-\mathbf 1_{[-b,-a]}
		\eneqn
		
		The case of a singleton is immediate. 
		
		\emph{($\implies$)}: Assume conversely that \(\phi\) is invertible. Choose a decomposition
		of \(\phi\) into non-zero multiples of indicator functions of consecutive interval pieces, with the number of pieces minimal. 
		
		Comparing the extremal non-zero germs in
		\[
		\phi\star\phi^{\star-1}=\delta_{0_L}
		\]
		shows that the coefficient of an extremal interval is a unit of \(\mathbb Z\), hence belongs to \(\{-1,1\}\). If the first non-zero piece is separated from the remaining support, we conclude by Lemma~\ref{lem:separated_interval_tail}. We may therefore assume that the first two pieces are adjacent (if the decomposition has only one piece, the assertion follows from \(\int_L\phi=\pm1\)).
		
		Let \(I_i\) be locally closed intervals with endpoints \(a_i,a_{i+1}\) in $L$, possibly with repetitions. Then
		\eqn
		\phi=\sum_{0\leq i\leq N-1}\alpha_i\mathbf 1_{I_i}.
		\eneqn
		
		Assume first that $\alpha_0, \alpha_1$ are non-zero, that the first term is a closed interval, that the second is closed at its lower end. Convolving with the inverse of \(\mathbf 1_{[a_0,a_1]}\), and multiplying by \(\alpha\in\{-1,1\}\), we obtain a constructible function of the form $	\delta_{0_L}+\eta$, where $\supp(\eta)\subset]0,+\infty[$. If \(0\notin\supp(\eta)\), Lemma~\ref{lem:separated_interval_tail} gives a contradiction. Otherwise, we will treat the case of a non-zero right-hand germ at \(0\) in the final argument below. 
		
		It remains to consider the situation where the first interval is closed but the second open at its lower end. Up to convolution by an invertible interval function, translation, reflection, and multiplication by \(-1\), we may find $a>0$, non-zero $\alpha,\beta$, and $\phi_1$ such that $\supp(\phi_1)\subset[a,+\infty[$
		and write
		\eqn
		\phi= \alpha\delta_{0_L} + \beta\mathbf 1_{[0,a)} + \phi_1
		\eneqn
		
		Let \(\psi=\phi^{\star-1}\). Denote by \(b_0\) the greatest point of \(\supp(\psi)\). The extremal-support relation allows us to assume, after translation, that \(b_0=0\).  There exist \(b_1<0\), non-zero integers \(\lambda,\mu\), and \(\psi_1\) with \(\supp(\psi_1)\subset]-\infty,b_1]\), such that
		\eqn
		\psi = \psi_1 + \mu\mathbf 1_{(b_1,0]} + \lambda\delta_{0_L},
		\eneqn
		
		Evaluating \(\phi\star\psi=\delta_{0_L}\) at \(0\), and comparing the
		one-sided values for a sufficiently small \(\varepsilon>0\), gives
		\[
		\alpha\mu=\beta\lambda,
		\qquad
		\alpha(\lambda+\mu)=1.
		\tag{3}
		\]
		Since all coefficients are integers and \(\beta\neq0\), the only possibilities for
		$(\alpha,\beta,\lambda,\mu)$ are $(-1,2,1,-2)$ or $(1,-2,-1,2)$. Thus, one may assume that $\alpha=-1,\beta=2$.
		
		Let us set \(A:=\mathbf 1_{[0,a)}\). A direct computation shows that
		\eqn
		(-\delta_{0_L}+2A)^{\star2}
		=
		\delta_{0_L}-4\mathbf 1_{[a,2a)}.
		\eneqn
		
		We deduce that
		\eqn
		\phi^{\star2} = \delta_{0_L}+\eta
		\eneqn
		for some constructible function \(\eta\) such that $\supp(\eta)\subset]0,+\infty[$
		
		Moreover, \(\eta\neq0\). Indeed, \(\phi\) is not supported at the origin, the
		extremal non-zero germ of \(\phi^{\star2}\) is non-zero. Now \(\phi^{\star2}\) is invertible because \(\phi\) is invertible. This
		contradicts Lemma~\ref{lem:separated_interval_tail}. 
	\end{proof}
	
	\section{Characterization of $\star$-invertibility}
	\label{sec:projections_and_invertibility}
	
	We establish uniqueness criteria for compactly supported $\star$-invertible sheaves and constructible functions
	in terms of their one-dimensional projections. We first show that a compactly supported sheaf whose projections are shifted Dirac sheaves is itself a shifted Dirac sheaf, and prove the analogous statement for constructible functions. As a consequence, we identify the convolution inverse of a compactly supported $\star$-invertible constructible sheaf with the dual of its antipodal transform. We also show that a compactly supported constructible sheaf is $\star$-invertible if and only if all its one-dimensional projections are, up to shifts, constant sheaves on non-empty bounded open or closed intervals. 
	
	These results will be essential to the complete characterization of compactly supported $\star$-invertible PL sheaves by convex factorization in Section~\ref{sec:convex_factorization_pl}.
	
	\subsection{Determination by one-dimensional projections}
	\label{sec:linear_projection_tools}
	
	We show that two compactly supported $\star$-invertible sheaves are isomorphic up to shift whenever their one-dimensional projections are isomorphic up to shifts. 
	
	We begin with the compatibility of convolution with linear direct images.
	\begin{lemma}
		\label{lemm:star_commutation_projection}
		Let \( F,G \in D_{\mathbb{R}-c,c}^b(k_E) \) and $u:E\rightarrow E'$ a linear map between finite-dimensional vector spaces. Then we have:
		\eqn
		\reim{u}(F \star G) \simeq \reim{u}(F) \star \reim{u}(G)
		\eneqn
		
	\end{lemma}
	
	\begin{proof}
		Let $\tilde{u}$ denote the linear map $u\times u: E\times E\to E'\times E'$. Using the K\"unneth formula, we have:
		\eqn
		\begin{array}{rcl}
			\reim{u}(F \star G) & := & \reim{u}\reim{s} (F \etens G)\\
			& \simeq & \reim{s}\reim{\tilde{u}} (F \etens G) \\
			& \simeq &  \reim{s}(\reim{u}F \etens \reim{u} G)\\
		\end{array}
		\eneqn
		
	\end{proof}
	
	The same compatibility holds for compactly supported constructible functions:  
	$$
	\eim{u}(\phi\star\psi)
	=
	\eim{u}\phi\star\eim{u}\psi.
	$$  
	
	\paragraph*{Sheaves.}
	Let us first relate one-dimensional projections to the Radon transform, using the notation of Section~\ref{reminder:radon_transform}.
	
	\begin{lemma}
		\label{lem:radon_projection_stalk}
		Let $F\in D^b(k_E)$ have compact support, and denote by $F$ also its extension by zero to $P$.
		Let $x\in L$ and $u\in\mathcal L^\times(E,L)$. Set $\mathcal V:=u^{-1}(x)$ and let $\xi\in P^*$ correspond to its projective closure.
		Then there are natural isomorphisms
		\[
		\Rad(F)_\xi[1-n]
		\simeq \rsect(\check{\xi};F)
		\simeq \rsect(\mathcal V;F)
		\simeq \reim{u}(F)_x
		\]
		
	\end{lemma}

	\begin{proof}
		This follows immediately from proper base change and compactness of $\supp(F)$, which identifies ordinary and compactly supported cohomology on $\mathcal V$.
		
	\end{proof}
	
	We next establish a microlocal criterion for a compactly supported sheaf to be supported at the origin. We first recall the following smooth-pullback criterion, due to Kashiwara--Schapira \cite[Prop.~5.4.5(ii)]{KS90}. Let $f:Y\to X$ be a smooth morphism, and let $f_d:Y\times_X T^*X\to T^*Y$ be the map induced by the transpose of its differential. 
	\begin{proposition}
		\label{prop:smooth_pullback_criterion}
		Let \(G\in D^b(k_Y)\). Then the following conditions are equivalent:
		\begin{enumerate}[label=(\roman*)]
			\item $SS(G)\subset f_d\bigl(Y\times_X T^*X\bigr)$
			\item Locally on \(Y\), there exists \(F\in D^b(k_X)\) such that $G\simeq f^{-1}F$
		\end{enumerate}
		
	\end{proposition}

	We have  
	\begin{lemma}
		\label{lem:dirac_sheaf_criterion}
		Let  \( G \in D^b_c(k_E) \). After identification of $T^*E$ with $E\times E^*$, assume that $SS(G)\subset \{(x,\xi)\in E\times E^*;\langle\xi,x\rangle=0\}$. Then $supp(G)\subset\{0_E\}$.
		
	\end{lemma}
	
	\begin{proof}
		Let us consider the natural projection $f:E\setminus\{0_E\}\to S_E$, where $S_E$ is the unit sphere. Let us observe that $f_d(E\setminus\{0_E\}\times_{S_E} T^*S_E)$ is exactly $\{(x,\xi)\in (E\setminus\{0_E\})\times E^*;\langle\xi,x\rangle=0\}$. 
		
		By Proposition~\ref{prop:smooth_pullback_criterion}, \(G|_{E\setminus\{0_E\}}\) is locally a pullback by \(f\). Hence its cohomology sheaves are locally constant along each fibre of \(f\). Since these fibres are connected and unbounded, compactness of \(\supp(G)\) forces \(G\) to vanish on every fibre. Therefore \(\supp(G)\subset\{0_E\}\).
		
	\end{proof}
	
	The following proposition will be used to characterize convolution invertibility in terms of one-dimensional projections.
	\begin{proposition}
		\label{prop:radon_transform}
		Let  \( F \in D^b_c(k_E) \). Assume that for every $u\in\mathcal L^\times(E,L)$, $\reim{u}(F)$ is a Dirac sheaf $\delta_{0_L}[\alpha_u]$ centered on $0_L$, with some shift $\alpha_u\in\mathbb{Z}$. Then $F$ is a Dirac sheaf up to a shift.
		
	\end{proposition}
	
	\begin{proof}
		The proposition is trivial for $n=1$. Assume $n\geq2$. View $E$ as an affine chart in $P$ and identify $T^*E\simeq E\times E^*$.
		
		Let $\xi\in P^*$ and let us compute the stalks $\Rad(F)_\xi$. Denote by $\mathcal{V}$ the affine hyperplane $\check{\xi}\cap E$, by $V$ its associated vector subspace. Choose $u\in\mathcal L^\times(E,L)$ with $\ker(u)=V$ and let us denote by $x\in L$ the element such that  $\opb{u}(\{x\})=\mathcal{V}$. Still denoting by $F$ the sheaf $F$ immersed in $P$, we have by Lemma~\ref{lem:radon_projection_stalk},
		$$
		\Rad(F)_\xi[1-n] \simeq \reim{u}(F)_x
		$$
		
		Then, by hypothesis, either $\xi\notin\widehat{\{0\}}_E$, and $\Rad(F)_\xi \simeq 0$, or $\xi\in\widehat{\{0\}}_E$, and $\Rad(F)_\xi \simeq k[\alpha_u+n-1]$. 
		
		Let $(x,\xi)\in (E\times E^*)\setminus I$, the local coordinate of $T^*P$.  We have $\langle\xi,x\rangle\neq 0$ and so $\xi\cap\widehat{\{0\}}_E=\emptyset$, (in fact, in projective coordinates $\xi=(\langle\xi,x\rangle:\xi_1:...:\xi_n)$). Hence $SS(\Rad(F))\cap \dT{}^{*}_\xi P^*=\emptyset$. 
		
		In the local chart $(E^* \times E) \setminus v(I)$ of $T^*P^*$, we identify the trace of $\dT{}^{*}_\xi P^*$ on this chart with $\{(\eta,x)\in(E^* \times E) \setminus v(I),\eta=\xi\}$. Using $\chi$, we get 
		$$
		SS(F)\cap ((E \times \{\xi\}) \setminus I)=\opb{\chi}(SS(\Rad(F))\cap ((\{\xi\}\times E) \setminus v(I)))=\emptyset
		$$
		
		Hence, we have $(x,\xi)\notin SS(F)$. Applying Lemma \ref{lem:dirac_sheaf_criterion}, we deduce that $F$ is supported by $\{0\}_E$ and there exist integers \(p_j\geq 0\), almost all zero, such that
		\[
		F\simeq \bigoplus_{j\in\mathbb Z} k_{\{0\}_E}^{\oplus p_j}[j]
		\] 
		By hypothesis the direct image of this sheaf by any $u\in\mathcal L^\times(E,L)$ is a Dirac sheaf, so that there exists $\alpha\in\mathbb{Z}$, such that $p_\alpha=1$ and $p_j=0$ for $j\neq \alpha$, i.e. $F\simeq\delta_{0_E}[\alpha]$.

	\end{proof}
	
	For $\star$-invertible sheaves, Proposition~\ref{prop:radon_transform} gives the following comparison criterion.  
	\begin{corollary}\label{cor:radon_transform}
		Let \( F,G \in D^b_c(k_E) \) be \(\star\)-invertible. Assume that for every $u\in\mathcal L^\times(E,L)$, $\reim{u}(F)\simeq\reim{u}(G)[\alpha_u]$, for some shift $\alpha_u\in\mathbb{Z}$. Then, up to a shift, $F\simeq G$.
	\end{corollary}
	
	\begin{proof}
		It is an immediate consequence of Lemma~\ref{lemm:star_commutation_projection} and Proposition~\ref{prop:radon_transform} applied to \(F\star G^{\star-1}\).
		
	\end{proof}
	
	\paragraph*{Constructible functions.}
	We now prove the analogue of Proposition~\ref{prop:radon_transform} for constructible functions.
	
	Let $P$ be the projective compactification of $E$, and let $\xi_\infty\in P^*$ correspond to $P\setminus E$. We extend compactly supported constructible functions on $E$ by zero to $P$.
	
	\begin{proposition}
		\label{prop:projection_dirac_constructible_functions}
		Let \(\phi\in\mathrm{CF}_c(E)\). If, for every $u\in\mathcal L^\times(E,L)$, $\eim{u}\phi=\delta_{0_L}$, then
		\[
		\phi=\delta_{0_E}
		\]
		
	\end{proposition}
	
	\begin{proof}
		The case \(\dim E=1\) is immediate. Assume \(n:=\dim E\geq2\). Let us first compute the Radon transform of \(\phi\). 
		
		Let \(\xi\in P^*\setminus\{\xi_\infty\}\). There exist \(u\in\mathcal L^\times(E,L)\) and \(t\in L\) such that $\check\xi\cap E=u^{-1}(\{t\})$. By hypothesis, \(u_!\phi=\delta_{0_L}\), therefore
		\[
		\Rad(\phi)(\xi)
		= (\eim{u}\phi)(t) = 
		\begin{cases}
			1,& \xi\in \widehat{0_E},\\
			0,& \xi\notin \widehat{0_E}
		\end{cases}
		\]
		Moreover, $\phi$ is zero at infinity, so that $\Rad(\phi)(\xi_\infty)=0$. Thus, we have
		\eq\label{eq:radon_projection_dirac_cf}
		\Rad(\phi)=\mathbf 1_{\widehat{0_E}}
		\eneq
		
		Let us set $\lambda_n:=\chi(P_{n-2}),\mu_n:=\chi(P_{n-1})$. Applying
		\cite[Th.~3.1 and Cor.~5.1]{Sch95}, we have
		\eq\label{eq:schapira_projection_dirac_cf}
		\Rad^*\Rad(\phi)
		=
		(\mu_n-\lambda_n)\phi
		+
		\lambda_n
		\left(
		\int_P\phi
		\right)\mathbf 1_P
		\eneq
		
		It remains to compute $\Rad^*(\mathbf 1_{\widehat{0_E}})$. For $x\neq0_E$, the intersection $\widehat x\cap\widehat{0_E}$ is an $(n-2)$-dimensional projective space, with Euler characteristic $\lambda_n$, whereas $\chi(\widehat{0_E})=\mu_n$. Hence
		
		\begin{align}\label{eq:dual_radon_hyperplane_cf}
			\Rad^*(\mathbf 1_{\widehat{0_E}})
			&= \lambda_n\mathbf 1_{P\setminus\{0_E\}}
			+ \mu_n\mathbf 1_{0_E} \notag\\
			&= (\mu_n-\lambda_n)\delta_{0_E}
			+ \lambda_n\mathbf 1_P
		\end{align}
		
		Combining \eqref{eq:radon_projection_dirac_cf}, \eqref{eq:schapira_projection_dirac_cf}, and \eqref{eq:dual_radon_hyperplane_cf}, we obtain
		\[
		(\mu_n-\lambda_n)\phi
		+
		\lambda_n
		\left(
		\int_P\phi
		\right)\mathbf 1_P
		=
		(\mu_n-\lambda_n)\delta_{0_E} + \lambda_n\mathbf 1_P
		\]
		Factoring the constant map $E\to \{0_L\}$ through any $u\in\mathcal L^\times(E,L)$ and recalling that $\phi$ is 0 at infinity, we have $\displaystyle\int_P\phi=1$, so that
		\eqn
		\phi=\delta_{0_E}
		\eneqn
		
	\end{proof}
	
	\subsection{Projection criteria and duality}
	\label{sec:projection_criteria}
	
	Combining the one-dimensional classifications with the preceding Dirac criteria yields a characterization of $\star$-invertibility and identifies the inverse by duality.
	
	\subsubsection{Constructible sheaves}
	We now identify the convolution inverse of an invertible sheaf with the dual of its antipodal transform.
	
	\begin{theorem}
		\label{th:description_inverse_star}
		Let $F$ be an object in $\Derb_{\mathbb{R}-c,c}(k_E)$ and assume $F$ is \(\star\)-invertible. Then the inverse of $F$ is the dual of its antipodal, $D_E(F^a)$. 
		
	\end{theorem}
	
	\begin{proof}
		Let $u\in\mathcal L^\times(E,L)$. Applying lemma \ref{lemm:star_commutation_projection}, $\reim{u}(F)$ is an invertible sheaf of a one dimensional vector space, and hence an open or closed interval. According to Theorem \ref{th:invertible_sheaf_dim_1}, is isomorphic, up to a shift, either to $k_{[a,b]}$ or $k_{]a,b[}$, whose inverse are $k_{]-b,-a[}[1]$, $k_{[-b,-a]}[1]$ respectively, that is to say $D_L(\reim{u}(F))^a$, (the dual in $\Derb_{\mathbb{R}-c}(k_L)$).
		
		Since $F$ has compact support, compatibility of duality with proper direct image gives, for every $u\in\mathcal L^\times(E,L)$,
		
		$$
		D_L(\reim{u}(F^a)) \simeq \reim{u}(D_E(F^a))
		$$
		
		Lemma~\ref{lemm:star_commutation_projection} then gives
		\begin{align*}
			\reim{u}(F \star D_E(F^a)) & \simeq \reim{u}(F) \star \reim{u}(D_E(F^a))\\
			& \simeq \reim{u}(F) \star D_L(\reim{u}(F^a)) \\
			& \simeq \reim{u}(F) \star D_L(\reim{u}(F))^a \\
			& \simeq\delta_{0_L}
		\end{align*}
		Applying the injectivity result of \ref{prop:radon_transform}, we get that $F \star D_E(F^a)\simeq\delta_{0_E}$.
		
	\end{proof}

	We can now state the projection criterion.  
	
	\begin{corollary}
		\label{cor:projection_inverse_star}
		Let $F$ be an object in $\Derb_{\mathbb{R}-c,c}(k_E)$. Then $F$ is \(\star\)-invertible if and only if, for every $u\in\mathcal L^\times(E,L)$, $\reim{u}(F)$ is invertible.
		
	\end{corollary}
	
	\begin{proof}
		\emph{($\implies$)}
		This follows immediately from Lemma~\ref{lemm:star_commutation_projection} and Theorem~\ref{th:invertible_sheaf_dim_1}.
		
		\emph{($\impliedby$)}
		Conversely, given Theorem \ref{th:description_inverse_star}, we have to prove that $D_E(F^a)\star F\simeq \delta_{0_E}$. Let $u\in\mathcal L^\times(E,L)$, and let us denote by $a_u,b_u$ the real numbers such that, let say, $\reim{u}(F)\simeq k_{[a_u,b_u]}$. We have the following isomorphisms:
		\begin{align*}
			\reim{u}(F\star D_E(F^a)) & \simeq \reim{u}(F)\star \reim{u}(D_E(F^a))\\
			& \simeq \reim{u}(F)\star D_L(\reim{u}(F^a))\\
			& \simeq k_{[a_u,b_u]} \star k_{]-b_u,-a_u[}[1]\\
			& \simeq \delta_{\{0_L\}}\\
		\end{align*}
		We conclude by Proposition \ref{prop:radon_transform}.
		
	\end{proof}
	
		\subsubsection{Constructible functions}
	
	We are in a position to prove the following theorem:
	
	\begin{theorem}
		\label{th:invertible_constructible_functions}
		Let \(\phi\in\mathrm{CF}_c(E)\). The following conditions are equivalent:
		\begin{enumerate}[label=(\roman*)]
			\item \(\phi\) is \(\star\)-invertible,
			
			\item for every \(u\in\mathcal L^\times(E,L)\), \(\eim{u}\phi\) is \(\star\)-invertible.
		\end{enumerate}
		Under these conditions, the inverse of \(\phi\) is $D_E(\phi^a)$
		
	\end{theorem}

	\begin{proof}
		Assume first that \(\phi\) is \(\star\)-invertible. Since proper direct image commutes
		with convolution, \(\eim{u}\phi\) is \(\star\)-invertible for every
		\(u\in\mathcal L^\times(E,L)\).
		
		Conversely, assume~(ii), and set $\psi:=D_E(\phi^a)$. Duality commutes with proper direct image, and hence
		\[
		\eim{u}\psi=D_L\bigl((\eim{u}\phi)^a\bigr).
		\]
		Also, by \ref{lemm:star_commutation_projection}, for every \(u\in\mathcal L^\times(E,L)\)
		\[
		\eim{u}(\phi\star\psi)
		=
		\eim{u}\phi\star \eim{u}\psi
		=
		\delta_{0_L}
		\]
		Applying Proposition \ref{prop:projection_dirac_constructible_functions}, we get
		\[
		\phi\star\psi=\delta_{0_E}
		\]
		which proves the result.
		
	\end{proof}

	\section{Radon characterization and convex factorization}
	\label{sec:arbitrary_dimension}
	
	We now turn to the geometric description of convolution invertibility. We first establish the passage from constructible functions back to invertible sheaves, then develop the role of convexity and the Radon transform, leading to convex factorization in the PL case.
	
	Throughout this section, we take $L=\mathbb R$ as the target line, so that $\mathcal L(E,L)=E^*$. We regard each $u\in E^*\setminus\{0\}$ as a one-dimensional projection $u:E\to L$.
	
	\subsection{Reconstruction from constructible functions}
	
	We first record the consequences of the preceding results for the constructible function associated with an invertible sheaf.
	
	\begin{proposition}
		\label{prop:euler_image_invertible_sheaf}
		Let \(F\in\Derb_{\mathbb R\text{-}c}(k_E)\) be compactly supported and invertible,
		and set $\phi:=\chi(F)\in\mathrm{CF}_c(E)$. Then \(\phi\) is \(\star\)-invertible, with inverse $D_E(\phi^a)$.
		Moreover, for every \(u\in E^*\setminus\{0\}\), there exist a bounded open or closed
		interval \(I_u\) and \(\varepsilon_u\in\{-1,1\}\) such that
		\[
		\eim{u}\phi=\varepsilon_u\mathbf 1_{I_u}.
		\]
		
	\end{proposition}

	\begin{proof}
		This follows immediately by applying the Euler-Poincar\'e morphism to
		Theorem~\ref{th:description_inverse_star} and
		Corollary~\ref{cor:projection_inverse_star}.
		
	\end{proof}
	
	Among convolution-invertible sheaves, the constructible function determines the sheaf up to an even shift. We have:  
	
	\begin{theorem}
		\label{th:reconstruction_from_euler_function}
		Let $F,G\in\Derb_{\mathbb R\text{-}c,c}(k_E)$ be $\star$-invertible. Then
		\[
		\chi(F)=\chi(G)
		\quad\Longleftrightarrow\quad
		F\simeq G[2d]
		\quad\text{for some }d\in\mathbb Z.
		\]
		
	\end{theorem}

	\begin{proof}
		Assume that $\chi(F)=\chi(G)$. For every $u\in E^*\setminus\{0\}$, compatibility with proper direct image gives
		\[
		\chi(\reim{u}F)
		=\eim{u}\chi(F)
		=\eim{u}\chi(G)
		=\chi(\reim{u}G)
		\]
		Since $F,G$ are $\star$-invertible, $\reim{u}F$ and $\reim{u}G$ are, up to a shift, constant sheaves on
		non-empty bounded open or closed intervals, by Corollary~\ref{cor:projection_inverse_star}.
		Since their associated constructible functions coincide, these projections are, up to an even shift, isomorphic as sheaves.
		Corollary~\ref{cor:radon_transform} therefore gives $F\simeq G[2d]$ for some $d\in\mathbb Z$.
		
		The converse is immediate.
		
	\end{proof}
	
	\subsection{Convexity and $\star$-invertibility}
	The projection criterion points to the role of convexity in convolution invertibility. This connection will be developed further in Section~\ref{sec:convex_factorization_pl}, with notably a factorization of invertible PL sheaves as convolutions of constant sheaves  on convex sets. 
	
	To get a first feel of that, let us consider constant sheaves on non-empty compact subanalytic set, for which convolution invertibility is equivalent to convexity of the underlying sets. Although this characterization follows from Theorem~\ref{th:invertible_sheaf_dim_1} and Aumann's theorem~\cite{Aumann36} (see also~\cite[S~1]{MS02}), we give a direct proof using the projective Radon transform, in the spirit of the geometric approach of this article.  
	
	\begin{theorem}
		\label{th:main_theorem}
		Let $S$ be a non-empty compact subanalytic set of $E$. Then the constant sheaf $k_S$ is \(\star\)-invertible if and only if $S$ is a compact convex set.
		
	\end{theorem}
	
	\begin{proof}
		\emph{($\implies$)}: We prove that \(S\) is convex by induction on \(n:=\dim(E)\).
		For \(n=1\), this is Theorem~\ref{th:invertible_sheaf_dim_1}. Assume \(n\ge 2\) and that the statement holds in dimension \(n-1\). 
		
		Let us assume that \(S\) is not convex. Then there exist two points $x,y\in S$ such that the segment $]x,y[\cap S=\emptyset$. Let us denote $\mathcal{L}$ the affine line that supports the segment $[x,y]$, $L$ the associated vector space of $\mathcal{L}$, by $H$ a complementary hyperplane of $L$ in $E$, and by $p\: E\rightarrow E$ the linear projection of kernel $L$ and image $H$. Set \(z:=p(x)=p(y)\) (equivalently \(z=\mathcal L\cap H\)).
		
		Consider the constructible function $\varphi:=\eim{p}\mathbf 1_S.$ Since the fibres of \(p\) are affine lines, \(\varphi(w)\) is the number of connected components of \(S\cap p^{-1}(w)\). In particular, $\varphi(z)\geq2$.
		
		Let \(A\subset H\) be an affine hyperplane containing \(z\). Then
		\(p^{-1}(A)\) is an affine hyperplane of \(E\) containing \(\mathcal L\),
		and hence \(S\cap p^{-1}(A)\neq\varnothing\). By
		Lemma~\ref{lemm:star_commutation_projection} and
		Theorem~\ref{th:invertible_sheaf_dim_1}, the direct image of \(k_S\)
		by any non-zero linear form on \(E\) is essentially, up to a shift, an open or closed interval. Since the trace on $S$ of fibers are compact and non-empty, they have non-zero \(H^0\), so that the shift is zero. Since higher cohomologies are zero, we get $\chi(S\cap p^{-1}(A))=1$, and so in the projective compactification of $H$, we get $\Rad(\varphi)(\xi)=1$ for every $\xi\in\widehat{\{z\}}$. We have
		\eqn
		(\Rad^*\Rad(\varphi))(z)= \chi(P_{n-2})
		\eneqn
		
		Moreover,
		\eqn
		\int_H\varphi = \int_E\mathbf 1_S = 1
		\eneqn
		
		We now apply Schapira's inversion formula~\eqref{eq:schapira_radon_inversion}
		on the projective compactification of \(H\), which has dimension
		\(n-1\):
		\eq
		\Rad^*\Rad(\phi)
		=
		\begin{cases}
			\phi,
			& n -1 \text{ odd},\\[0.3em]
			-\phi+ \mathbf 1_P,
			& n -1 \text{ even}
		\end{cases}
		\eneq
		Using $\chi(P_{n-2})=\frac{1+(-1)^{n}}{2}$, we get $\phi(z)=1$ for $n$ odd or even. This leads to a contradiction, hence \(S\) is convex.
		
		\emph{($\impliedby$)}: let $d$ be the dimension of the affine subspace generated by the convex hull of $S$, then the sheaf $k_S$ is invertible with inverse $k_{I(S)}[d]$.

	\end{proof}
	
	\subsection{Radon characterization of invertibility}
	
	We now express the projection criterion geometrically in the dual projective space.
	
	Let $P$ be the projective compactification of $E$ , let $H_\infty:=P\setminus E$ , and let $\xi_\infty\in P^*$
	correspond to $H_\infty$ . Denote by $H_\infty^*$ the projective dual of $H_\infty$ .  
	Restriction to $H_\infty$ defines a locally trivial fibration  
	$$
	\rho:P^*\setminus\{\xi_\infty\}\longrightarrow H_\infty^*
	$$  
	where $\rho(\xi)$ is the point corresponding to the hyperplane $\check\xi\cap H_\infty$ of $H_\infty$. Thus each fibre parametrizes a family of parallel affine hyperplanes. This is the hyperplane case of the natural fibration from affine to linear Grassmannians, sending an affine subspace to its direction.

	We identify $H_\infty^*$ with $(E^*\setminus\{0\})/\mathbb R^\times$. Here $[u]$ denotes the class of the non-zero linear form $u$ modulo multiplication by a non-zero scalar, corresponding to the trace at infinity of $\ker u$, and the fibre $\rho^{-1}([u])$ parametrizes the hyperplanes $\{x\in E;\ u(x)=t\}$ as $t$ varies in $\mathbb R$.
	
	For $u\in E^*\setminus\{0\}$, let
	\[
	i_u:\mathbb R\xrightarrow{\sim}\rho^{-1}([u])
	\]
	send $t\in\mathbb R$ to the point of $P^*$ corresponding to the projective closure of the affine hyperplane $\{x\in E;\ u(x)=t\}$. For $\lambda\in\mathbb R^\times$, these parametrizations satisfy
	\[
	i_{\lambda u}(\lambda t)=i_u(t)
	\]
	
	In the following, for a function $\phi\in\mathrm{CF}_c(E)$, we also denote by $\phi$ its extension by zero to $P$ .  
	
	\begin{lemma}
		\label{lem:projector_radon_constructible_functions}
		Let \(\phi\in\mathrm{CF}_c(E)\). Then, for every \(u\in E^*\setminus\{0\}\),
		\[
		i_u^*\Rad(\phi)=\eim{u}\phi
		\]
		
	\end{lemma}

	\begin{proof}
		For \(t\in \mathbb{R}\), one has
		\[
		\Rad(\phi)(i_u(t))
		=
		\int_{\widecheck{i_u(t)}}\phi
		=
		\int_{u^{-1}(t)}\phi
		=
		(\eim{u}\phi)(t)
		\]
		
	\end{proof}
	
	For an invertible function, convolution with a sufficiently large closed ball makes all projected intervals closed. This motivates the following definition.
	
	\begin{definition}
		A subanalytic set $\Omega\subset P^*\setminus\{\xi_\infty\}$ is called $\rho$-trivial if, for every $u\in E^*\setminus\{0\}$, the set $I_u:=i_u^{-1}(\Omega)$ is a non-empty bounded closed
		interval. It is called compactly $\rho$-trivial if, in addition, $\xi_\infty\notin\overline{\Omega}$.
	\end{definition}

	We have the following useful geometrical caracterisation of  $\star$ -invertible constructible functions:  
	\begin{proposition}
		\label{prop:main_theorem_cf}
		Let \(\phi\in\mathrm{CF}_c(E)\). The following conditions are equivalent:
		\begin{enumerate}[label=(\roman*)]
			\item \(\phi\) is \(\star\)-invertible.
			\item There exist a sufficiently large closed ball \(B\subset E\),
			a compactly \(\rho\)-trivial set \(\Omega\subset P^*\), and
			\(\varepsilon\in\{-1,1\}\) such that
			\(\Rad^*(\mathbf 1_\Omega)\) is constant in a neighborhood of \(H_\infty=P\setminus E\) and, in \(\mathrm{CF}(E)/\mathbb Z\mathbf 1_E\),
			\[
			\phi\star\mathbf 1_B = \varepsilon(-1)^{n-1}\Rad^*(\mathbf 1_\Omega)|_E
			\]
		\end{enumerate}
	\end{proposition}

	\begin{proof}
		\emph{(\(\implies\))} Assume that \(\phi\) is \(\star\)-invertible. By Proposition~\ref{prop:invertible_constructible_functions_line}, for every \(u\in E^*\setminus\{0\}\), there exists a bounded open or closed interval \(I_u\) such that \(\eim{u}\phi=\pm\mathbf 1_{I_u}\).
		Obviously, the convolution of the indicator function of an open interval by the indicator function of a larger closed interval is supported by a closed interval. Hence, choosing a sufficiently large closed ball \(B\subset E\) and setting
		\[
		\widetilde\phi:=\phi\star\mathbf 1_B
		\]
		we may assume that, for every \(u\in E^*\setminus\{0\}\), \(\eim{u}\widetilde\phi\) is, up to sign, the indicator function of a bounded closed interval \(I_u\). We still denote by \(\widetilde\phi\) its extension by zero to \(P\).
		Set $\Omega:=\{\xi\in P^*;\ \Rad(\widetilde\phi)(\xi)\neq0\}$. By Lemma~\ref{lem:projector_radon_constructible_functions}, $i_u^*\Rad(\widetilde\phi) = \eim{u}\widetilde\phi=
		\varepsilon_u\mathbf 1_{I_u}$, where \(\varepsilon_u\) is a sign depending on \(u\). Applying
		Lemma~\ref{lem:radon_integrals_fibres} below, we get that \(\varepsilon_u=\varepsilon\) does not depend on \(u\). Hence
		\[
		\Rad(\widetilde\phi)=\varepsilon\mathbf 1_\Omega
		\]
		and \(\Omega\) is \(\rho\)-trivial.
		
		Since $\widetilde\phi$ has compact support, $\Rad(\widetilde\phi)$ vanishes near $\xi_\infty$, so $\Omega$ is compactly $\rho$-trivial. Schapira's inversion formula then shows that $\Rad^*(\mathbf 1_\Omega)$ is constant near $H_\infty$ and gives~(ii).
		
		\emph{(\(\impliedby\))} Conversely, assume~(ii), and set
		\(\widetilde\phi:=\phi\star\mathbf 1_B\). By Schapira's inversion
		formula applied on \(P^*\), the functions \(\Rad(\widetilde\phi)\) and
		\(\varepsilon1_\Omega\) differ by a constant. Since both vanish in a neighborhood
		of \(\xi_\infty\), this constant is zero and hence $\Rad(\widetilde\phi)=\varepsilon1_\Omega$.
		
		Set \(\psi:=D_E(\widetilde{\phi}^a)\). Since duality commutes with proper direct image and
		\(u\) commutes with the antipodal maps, it follows that, for every \(u\in E^*\setminus\{0\}\),
		\eqn 
		\eim{u}\psi = D_L\bigl((\eim{u}\widetilde\phi)^a\bigr) = D_L\bigl((\varepsilon\mathbf 1_{I_u})^a\bigr) = 
		(\eim{u}\widetilde\phi)^{\star-1}
		\eneqn
		Therefore, Proposition
		\ref{prop:projection_dirac_constructible_functions} then gives
		\(\widetilde\phi\star\psi=\delta_{0_E}\). Thus \(\widetilde\phi\) is \(\star\)-invertible. Since \(\mathbf 1_B\) is \(\star\)-invertible, so is \(\phi\).
		
	\end{proof}

	\begin{lemma}
		\label{lem:radon_integrals_fibres}
		Let \(\phi\in\mathrm{CF}_c(E)\), extended by zero to \(P\), and set \(\psi:=\Rad(\phi)\). Then
		\eqn
		\widetilde\psi:=\eim{\rho}\psi\equiv0
		\quad\text{in }\mathrm{CF}(H_\infty^*)/\mathbb Z\mathbf 1_{H_\infty^*}
		\eneqn
		
	\end{lemma}

	\begin{proof}
		Since \(\psi\) vanishes near \(\xi_\infty\), \(\rho\) is proper on its support. For \(\eta\in H_\infty\), functoriality of Euler integration gives
		\eqn
		\Rad_{H_\infty^*}(\widetilde\psi)(\eta)
		=\int_{\{\ell\in H_\infty^*;\,\eta\in\check\ell\}}
		\left(\int_{\rho^{-1}(\ell)}\psi\right)
		=\int_{\widehat\eta}\psi
		=(\Rad^*\psi)(\eta).
		\eneqn
		where $\Rad_{H_\infty^*}$ denotes the Radon transform on $H_\infty^*$. By Schapira's inversion formula \cite[Cor.~5.1]{Sch95}, we have modulo constant functions on $H_\infty$:
		\eqn
		\Rad_{H_\infty^*}(\widetilde\psi) = (\Rad^*\Rad(\phi))|_{H_\infty} \equiv(-1)^{n-1}\phi|_{H_\infty}=0
		\eneqn
		Applying the same inversion formula for $\widetilde\psi$ on \(H_\infty^*\) yields \(\widetilde\psi\equiv0\) modulo constant functions.
		
	\end{proof}
	
	\subsection{Convex factorization in the PL case}
	\label{sec:convex_factorization_pl}

	\subsubsection{Continuity of the endpoint functions}
		
		We are now in position to prove invertible PL constructible functions are convolution of indicator functions of a closed bounded convex set and a relatively open convex set. We will first need a few lemmas. Let us fix some notations.
		
		Let $\phi\in\mathrm{CF}_c(E)$ be invertible, extended by zero to $P$. By Proposition~\ref{prop:main_theorem_cf}, after convolution with a closed ball and a change of sign if necessary, we may assume $\Rad(\phi)=\mathbf 1_\Omega$ for a compactly $\rho$-trivial set $\Omega\subset P^*\setminus\{\xi_\infty\}$.
		
		Let \(W\subset P^*\) be an affine chart not containing \(\xi_\infty\), and set \(\Sigma:=W\cap\widehat{\{0_E\}}\). After convolution of $\phi$ with a sufficiently large ball, $\Omega$ is $\rho-$trivial and we may assume that, for each $\eta\in H_\infty^*$, \(\rho^{-1}(\eta)\cap\Omega\) is a closed interval $I_\eta$ and contains $0$.
		
		Choose affine coordinates \(W\simeq\Sigma\times\mathbb R\) in which \(\Sigma\) is the zero section and the fibres of \(\rho\) are vertical. Using the identification above, define the subanalytic endpoint functions \(a,b:\Sigma\to\mathbb R\) by
		\eqn
		I_{x}=[a(x),b(x)],\qquad x\in\Sigma.
		\eneqn

		Choose a neighborhood \(V\) of \(\xi_\infty\) such that \(\check\xi\cap\supp(\phi)=\varnothing\) for every \(\xi\in V\), and fix \(K>0\) such that every affine line of absolute slope greater than \(K\) has its point at infinity in \(V\).
		
		\begin{lemma}\label{lem:continuity_upper_endpoint_dim_2}
			Assume $P$ is of dimension $2$. Then the endpoint functions \(a,b\) are \(K\)-Lipschitz.
		\end{lemma}
		
		\begin{proof}
			
			By Lemma~\ref{lem:radon_integrals_fibres}, applied with centre \(\xi\in V\), the integral of \(1_\Omega\) is constant on the projective lines through \(\xi\). And this constant does not depend on $\xi$, by simple evaluation on the line joining \(\xi\) to \(\xi_\infty\). Hence, for every projective line \(\ell\) meeting \(V\), $\int_\ell\mathbf 1_\Omega=\pm 1$. Let us fix this constant to be $1$.
			
			Fix \(x\in\Sigma\). By subanalyticity and local boundedness, the endpoint functions are analytic on either side of \(x\), after shrinking the neighborhood, and have finite one-sided limits. If \(b\) were discontinuous at \(x\), a sufficiently vertical affine line through \((x,b(x))\), would have a section of Euler characteristic \(0\) or at least \(2\). Indeed, distinguishing the positions of \(b(x)\) relative to its one-sided limits gives either half-open sections or a closed disconnected section, leading to a contradiction. Consequently, every affine line whose point at infinity belongs to \(V\) meets \(\Omega\) in a compact subanalytic set of Euler characteristic \(1\), hence in an interval.
			
			Assume tha \(b\) is not \(K\)-Lipschitz. There would be arbitrarily close graph points $(x_{-},y_{-}), (x_{+},y_{+})$ whose generated line has absolute slope greater than \(K\), let say with $y_{-}<y_{+}$. One can find a sufficiently small $\varepsilon>0$, such that the line $l$ through $(x_{-},y_{-}+\varepsilon), (x_{+},y_{+})$ has slope strictly greater than $K$. But $(x_{-},y_{-}+\varepsilon)$ lies in the segment drawn by $(x_{+},y_{+})$ and the intersection of $l$ with the zero section, while $(x_{-},y_{-}+\varepsilon)\notin\Omega$. This contradicts the interval property. 
			
			The same argument applies to \(a\).
		\end{proof}
		
		We now deduce continuity in arbitrary dimension.
		
		\begin{lemma}\label{lem:continuity_upper_endpoint}
			Assume $P$ is of dimension $n\geq2$. The endpoint functions \(a,b:\Sigma\to\mathbb R\) are continuous.
		\end{lemma}
		
		\begin{proof}
		
			With the same argument as in the preceding lemma, for every projective line \(\ell\) meeting \(V\), $\int_\ell\mathbf 1_\Omega=1$.
		
			To each line $\ell$ transversal to $\Sigma$, we attach the plane $\Pi_\ell$, generated by $\ell$ and its projection on $\Sigma$ with respect to the decomposition $W\simeq\Sigma\times\mathbb R$. In particular, if the point at infinity of an affine line \(\ell\subset W\) belongs to \(V\), its intersection with \(\Omega\) is a compact interval. Indeed, we apply Lemma \ref{lem:continuity_upper_endpoint_dim_2} to the plane $\Pi_\ell$.
			
			Let us fix a norm $\|.\|$ on $\Sigma$. Since \(V\) is a neighborhood of the vertical direction, there exists \(K>0\) such that every line of slope greater than \(K\) has its point at infinity in \(V\).
			
			Suppose that \(b\) is discontinuous at \(x\). Choose \(\varepsilon>0\) and \(z_j\to x\) such that \(|b(z_j)-b(x)|>\varepsilon\). If the \(z_j\) lie on a fixed line, the planar case gives a contradiction. Otherwise, fix a line \(\ell\subset\Sigma\) through \(x\) and let \(y_j\) be the projections of \(z_j\) onto \(\ell\), parallel to a fixed complementary subspace. Let \(\ell_j\) be the line through \((z_j,b(z_j))\) and
			\((y_j,b(y_j))\), and set \(\Pi_j:=\operatorname{Aff}(z_j,y_j)\times\mathbb R\).
			
			By the planar case, \(b(y_j)\to b(x)\), hence, for large \(j\),
			\eqn
			\frac{|b(z_j)-b(y_j)|}{\|z_j-y_j\|}
			\geq\frac{\varepsilon}{2\|z_j-y_j\|}
			\longrightarrow+\infty
			\eneqn
			We consider the function $b$ restricted to the plane $\Pi_j$. We get a line $b(z_j),b(y_j)$ for infinitesimaly closed points with a slope stricly greater than \(K\), contradicting Lemma \ref{lem:continuity_upper_endpoint_dim_2}.
		\end{proof}
		
		\subsubsection{Dual Radon transforms of convex cones}
		
		We will make use of a few geometrical lemmas. We will first compute the Radon transform of a cone with base an $(n-1)$-dimensional affine convex set. Let us fix some notations.

		Let \(H_O:=\widehat{\{O\}}\), \(\Sigma:=H_O\cap W\). Let \(B\subset\Sigma\) be compact, subanalytic and convex of dimension \(n-1\), and let \(C\in W\setminus\Sigma\) have its vertical projection in \(\mathring B\). 
		\eqn
		\Gamma:=\operatorname{Conv}(B\cup\{C\}),\qquad
		\widetilde\Gamma:=\widecheck\Gamma\setminus\widecheck B,\qquad
		K:=\overline{\widetilde\Gamma},\qquad
		\widetilde B:=K\cap\widecheck C
		\eneqn
		
		\begin{lemma}\label{lem:radon_convex_cap_factorization}
			The sets \(K,\widetilde B\subset E\) are compact convex sets of dimensions \(n,n-1\), respectively, and
			\eqn
			\Rad^*(\mathbf 1_{\Gamma\cup H_O})|_E
			=\mathbf 1_{-K}\star\mathbf 1_{\mathring{\widetilde B}}+\chi(P_{n-2})\mathbf 1_E
			\eneqn
		\end{lemma}
		
		\begin{proof}
			First, we observe that the complementary of set-Radon transform of a compact convex set with non-empty interior is a bounded convex set in a suitable affine chart. Let \(S\) denote such a set, and consider two suitably normalized affine forms \(\phi,\psi\) with \(\phi|_S,\psi|_S<0\). Then the segment \(\lambda\phi+(1-\lambda)\psi\), \(\lambda\in[0,1]\), remains strictly negative on \(S\).
			
			Applied to $B$, and restricting to the projective subspace $\widecheck C$ of hyperplanes through \(C\), this shows that $\widecheck C\setminus\widecheck B
			=\widecheck C\cap(\widecheck B)^c$ is convex in the induced affine chart. The hypothesis on the vertical projection of \(C\) makes its closure bounded. And we have 
			\eqn
			\mathring{\widetilde B}=\widecheck C\setminus\widecheck B
			\eneqn
			
			Each such hyperplane $H$ meets \(\Gamma\) only at \(C\) and the pencil it generates with \(H_O\) identifies the hyperplanes meeting \(\Gamma\) with no intersection with \(B\) and containing $H\cap H_O$ with the segment joining the dual point of $H$ to \(O\), with \(O\) omitted. Hence
			\eqn
			K=\operatorname{Conv}(\{O\}\cup\widetilde B)
			\eneqn
			Consequently, we have:
			\eqn
			\widetilde\Gamma=\mathring K\cup\mathring{\widetilde B}
			\eneqn
			
			We now pass to indicator functions. Since \(\Gamma\cap H_O=B\) and every non-empty hyperplane section of \(\Gamma\) or \(B\) is compact and convex, additivity gives:
			\eqn
			\begin{aligned}
				\Rad^*(\mathbf 1_{\Gamma\cup H_O})|_E
				& = \Rad^*(\mathbf 1_{H_O})|_E + \Rad^*(\mathbf 1_{\Gamma\setminus B})|_E\\
				& = \chi(P_{n-2})\mathbf 1_E + (-1)^{n-1}\delta_O+\mathbf 1_{\widetilde\Gamma}\\
			\end{aligned}
			\eneqn
			Euler integration on sections parallel to the base gives:
			\eqn
			\mathbf 1_{\mathring K}\star\mathbf 1_{-\widetilde B}
			=(-1)^{n-1}\bigl(\mathbf 1_{-K}-\mathbf 1_{-\widetilde B}-\delta_O\bigr).
			\eneqn
			Hence:
			\eqn
			\mathbf 1_{\widetilde\Gamma}\star\mathbf 1_{-\widetilde B}
			=(-1)^{n-1}\bigl(\mathbf 1_{-K}-\mathbf 1_{-\widetilde B}\bigr)
			\eneqn
			Hence:
			\eqn
			(\Rad^*(\mathbf 1_{\Gamma\cup H_O})|_E-\chi(P_{n-2})\mathbf 1_E)\star \mathbf 1_{-\widetilde B} = (-1)^{n-1}\mathbf 1_{-K}
			\eneqn
			and so
			\eqn
			\begin{aligned}
				\Rad^*(\mathbf 1_{\Gamma\cup H_O})|_E 
				& = (-1)^{n-1}\mathbf 1_{-K}\star ((-1)^{n-1} \mathbf 1_{\mathring{\widetilde B}}) + \chi(P_{n-2})\mathbf 1_E\\
				& = \mathbf 1_{-K}\star  \mathbf 1_{\mathring{\widetilde B}} + \chi(P_{n-2})\mathbf 1_E
			\end{aligned}
			\eneqn
		\end{proof}
		
		We will need this geometrical Lemma:
		
		\begin{lemma}\label{lem:radon_convolution}
			Let \(\phi_1,\phi_2\in\mathrm{CF}_c(E)\) be invertible, with closed fibre intervals. Their associated endpoint functions add under convolution:
			\eqn
			a_{\phi_1\star\phi_2}=a_{\phi_1}+a_{\phi_2},
			\qquad
			b_{\phi_1\star\phi_2}=b_{\phi_1}+b_{\phi_2}
			\eneqn
		\end{lemma}
		
		\begin{proof}
			By Lemma~\ref{lem:projector_radon_constructible_functions}, this follows from compatibility of direct image with convolution and the addition of endpoints under convolution of closed intervals.
		\end{proof}
		
		\subsubsection{Convex factorization of  \(\star\)-invertible constructible functions}
		
		We are now in a position to prove the following proposition, which leads to the final theorem. We will consider below the endpoint function $b$, but the following also applies to the endpoint function $a$. Assume \(\dim E=n\geq2\), let \(b:\Sigma\to\mathbb R_{\geq0}\) be compactly supported, piecewise affine and continous, and set
		\eqn
		\Omega:=H_O\cup\{(s,t)\in W;\ 0\leq t\leq b(s)\}.
		\eneqn
		
		Let \(\mathcal T\) be a locally finite triangulation of \(\Sigma\) adapted to \(b\). We refer to \cite{GCU23} for the following. The vertex hat \(h_v\) associated with a vertex \(v\) is the continuous function affine on each simplex, equal to \(1\) at \(v\) and \(0\) at all other vertices. Its support is the union of the simplices containing \(v\), called the vertex star of \(v\). Since $b$ is continuous and affine on each
		simplex of $\mathcal T$, the interpolation formula of~\cite[\S~2.1.2, (5)--(6)]{GCU23} gives
		\eq\label{nodal_decomposition}
		b=\sum_v b(v)h_v
		\eneq
		The sum is finite by compact support of $b$ and local finiteness of $\mathcal T$.
		
		When $\dim E=2$, each non-zero term $b(v)h_v$ is the upper-boundary function of the triangle with vertex $(v,b(v))$ and base the interval joining the two neighbouring vertices of the triangulation. The following lemma generalizes this elementary construction to arbitrary dimensions.
		
		We will need the following terminolgy introduced in ~\cite{GCU23}. For a vertex \(v\) with \(b(v)>0\), we consider its vertex star \(\mathcal S_v\subset\Sigma\), and we choose a small compact convex polytope \(Q_v\), with \(v\in\mathring Q_v\), contained in its visibility kernel, namely the set of points \(q\in\mathcal S_v\) such that \([q,z]\subset\mathcal S_v\) for every \(z\in\mathcal S_v\).

		\begin{lemma}\label{lem:convex_cone_decomposition}
			Let $b:\Sigma\to\mathbb R_{\geq0}$ be compactly supported, piecewise affine and continuous. Then
			\eqn
			b=\sum_{i=1}^N m_i t_i,\qquad m_i\in\mathbb Z,
			\eneqn
			where each $t_i$ is the upper-boundary function, extended by zero, of a compact convex cone $\Gamma_i$ with polyhedral base in $\Sigma$ and vertex projecting into the interior of its base in $\Sigma$.
		\end{lemma}
		
		\begin{proof}
			The continuity assumption ensures that the vertex decomposition above applies. Compact support and local finiteness of $\mathcal T$ make this decomposition finite.
			
			Fix a vertex $v$ with $b(v)>0$, and let $\sigma_1,\ldots,\sigma_r$ be the maximal simplices containing $v$. Set $B_i:=\operatorname{Conv}(Q_v\cup\sigma_i)$. The visibility-kernel property gives $B_i\subset\mathcal S_v$ and $\bigcup_i B_i=\mathcal S_v$. For each non-empty $J\subset\{1,\ldots,r\}$, set
			\[
			B_J:=\bigcap_{i\in J}B_i,
			\qquad
			\Gamma_{v,J}:=\operatorname{Conv}
			\bigl((B_J\times\{0\})\cup\{(v,b(v))\}\bigr).
			\]
			Since $Q_v\subset B_J$, the vertex projects into the interior of $B_J$.
			
			Let $t_{v,J}$ be the upper-boundary function of $\Gamma_{v,J}$, extended by zero. The inclusion–exclusion formula gives:
			\[
			b(v)h_v = \sum_{\varnothing\ne J\subset\{1,\ldots,r\}} (-1)^{|J|+1}t_{v,J}.
			\]
			Summing over $v$ gives the required decomposition.
		\end{proof}
		
		\begin{proposition}\label{prop:factorization_piecewise_linear_strip}
			There exist non-empty bounded convex sets \(C,D\subset E\),
			with \(C\) closed and \(D\) relatively open, such that
			\eqn
			\bigl(\Rad^*(\mathbf 1_\Omega)|_E-\chi(P_{n-2})\mathbf 1_E\bigr)
			=\pm\mathbf 1_C\star\mathbf 1_D
			\eneqn
		\end{proposition}
		
		\begin{proof}			
			
			By Lemma~\ref{lem:convex_cone_decomposition}, write $b=\sum_{i=1}^N m_i t_i$, with associated cones $\Gamma_i$. Then, by Lemma~\ref{lem:radon_convex_cap_factorization}, for each \(\Gamma_i\), there exists an invertible constructible function 
			\eqn
			\phi_i=(-1)^{n-1}\mathbf 1_{A_i}\star \mathbf 1_{B_i}
			\eneqn
			with \(A_i,B_i\) bounded closed and relatively open convex sets, such that \(\Rad(\phi_i)=\mathbf 1_{H_O\cup\Gamma_i}\). Hence, $a_{\phi_i}=0$, $b(\phi_i)=t_i$. 
			
			Set \(b_+:=\sum_{m_i>0}m_it_i\), \(b_-:=\sum_{m_i<0}(-m_i)t_i\), and
			\(g:=\mathop{\star}_{i=1}^N\phi_i^{\star m_i}\).
			For $s\in\Sigma$, choose $u_s\in E^*\setminus\{0\}$ with $[u_s]=\rho(s)$ so that $i_{u_s}$ agrees with the chosen fibre coordinate. Since $b=b_+-b_-\geq0$, the interval convolution formulas give
			\[
			\eim{u_s}g
			=\mathbf 1_{[0,b_+(s)]}\star
			\bigl(\mathbf 1_{[0,b_-(s)]}\bigr)^{\star-1}
			=\mathbf 1_{[0,b(s)]}
			\]
			
			Hence \(\Rad(g)=\mathbf 1_\Omega\) and \(\int_Eg=1\). Schapira's inversion formula gives
			\eqn
			\Rad^*(\mathbf 1_\Omega)|_E-\chi(P_{n-2})\mathbf 1_E
			=(-1)^{n-1}g
			\eneqn
			
			Using the convolution and inverse formulas for convex indicators, we can group the factors as \(g=\mathbf 1_C\star\mathbf 1_D\), whith \(C,D\) bounded closed and relatively open convex polytopes respectively.
		\end{proof}
		
		In the PL setting, we take the regularizing ball with respect to a polyhedral norm. We can now prove the PL factorization theorem.
		
		\begin{theorem}
		\label{th:main_theorem_cf}
		Let \(\phi\in\mathrm{CF}_c(E)\) and assume $\phi$ is PL. The following conditions are equivalent:
		\begin{enumerate}[label=(\roman*)]
			\item \(\phi\) is \(\star\)-invertible.
			
			\item There exists a pair of non-empty bounded convex polytopes \(C,D\subset E\),
			with \(C\) closed and \(D\) relatively open, such that  
			\eqn
			\phi = \pm 1_{C} \star 1_{D}
			\eneqn
		\end{enumerate}
	\end{theorem}
	\begin{proof}
		$(\impliedby)$ is proved in \ref{prop:main_theorem_cf}. Let us prove $(\implies)$.
		
		Assume that \(\phi\) is invertible. After normalization, its Radon transform is, up to sign, the indicator function of a closed \(\rho\)-trivial set \(\Omega\). 
		
		Choose finitely many affine charts \(\Sigma_j=H_O\cap W_j\) and compact polytopes \(Q_j\subset\Sigma_j\) whose interiors cover \(H_O\). 
		
		We proceed successively on each \(Q_j\). In \(\Sigma_j\), let \(a\leq0\leq b\) be the endpoint functions of $\phi$. Choose a locally finite triangulation adapted to them and decompose $a$ and $b$ with regards to \ref{nodal_decomposition}. We then consider the terms at the vertices of the simplices meeting \(Q_j\). Their sums \(A_j,B_j\) are PL and compactly supported in
		\(\Sigma_j\), with
		\[
		a\leq A_j\leq0\leq B_j\leq b
		\]
		Applying Proposition~\ref{prop:factorization_piecewise_linear_strip} to \(-A_j\) and \(B_j\), gives a $\star$-invertible function \(g_j\) in $E$ whose Radon transform has endpoint functions \([A_j,B_j]\). Convolution \(\phi\star g_j^{\star-1}\) replaces the current endpoint function by $a-A_j,b-B_j$. These are still PL, and equal \(\{0\}\) on \(Q_j\). After finitely many steps, all fibres are therefore \(\{0\}\). Proposition~\ref{prop:projection_dirac_constructible_functions} identifies the residual function with \(\delta_{0_E}\). Hence, 
		\eqn
			\phi = \mathop{\star}_j g_j
		\eneqn
		Regrouping the convex factors proves the assertion.
	\end{proof}			
	
	\subsubsection{Factorization of PL sheaves}
	
	Combining Theorem~\ref{th:main_theorem_cf} with
	Theorem~\ref{th:reconstruction_from_euler_function}
	yields the following factorization of invertible PL sheaves.
	
	\begin{corollary}\label{th:factorization_invertible_pl_sheaves}
		Let \(F\in\Derb_{\mathbb R\text{-}c,c}(k_E)\) be PL.
		The following conditions are equivalent:
		\begin{enumerate}[label=(\roman*)]
			\item \(F\) is \(\star\)-invertible.
			
			\item There exist non-empty bounded convex polytopes \(C,D\subset E\), with \(C\) closed and \(D\) relatively open, and \(m\in\mathbb Z\), such that
			\eqn
			F\simeq k_C\star k_D[m]
			\eneqn
		\end{enumerate}
	\end{corollary}
	
	\begin{proof}
		$(\impliedby)$ is immediate since constant sheaves on non-empty bounded closed or relatively open convex sets
		are invertible. Let us prove $(\implies)$.
		
		Assume that \(F\) is invertible. Then \(\chi(F)\) is compactly supported, PL and invertible. By Theorem~\ref{th:main_theorem_cf}, there exist non-empty bounded convex polytopes \(C,D\), with \(C\) closed, \(D\) relatively open, and \(e\in\{0,1\}\), such that
		\eqn
		\chi(F)=(-1)^e\mathbf 1_C\star\mathbf 1_D
		\eneqn
		Set \(G:=k_C\star k_{D}[e]\). Then \(G\) is compactly supported and invertible,
		and \(\chi(G)=\chi(F)\). By Theorem~\ref{th:reconstruction_from_euler_function}, there exists \(d\in\mathbb Z\) such that \(F\simeq G[2d]\). Hence
		\eqn
		F\simeq k_C\star k_D[2d+e]
		\eneqn
		Taking $m=2d+e$ proves~(ii).
	\end{proof}

\printbibliography
\end{document}